\documentclass[11pt]{article}
\usepackage[a4paper]{anysize}\marginsize{3.5cm}{3.5cm}{1.3cm}{2cm}
\pdfpagewidth=\paperwidth \pdfpageheight=\paperheight
\usepackage{amsfonts,amssymb,amsthm,amsmath,eucal}
\usepackage{pgf}
\usepackage{bbm,array}
\theoremstyle{plain}
\newtheorem{thm}{Theorem}[section]
\newtheorem{theorem}[thm]{Theorem}

\newtheorem{lemma}[thm]{Lemma}
\newtheorem{proposition}[thm]{Proposition}
\newtheorem{corollary}[thm]{Corollary}

\theoremstyle{definition}
\newtheorem{definition}[thm]{Definition}
\newtheorem{remark}[thm]{Remark}
\newtheorem{example}[thm]{Example}

\newtheorem{thevarthm}[thm]{\varthmname}

\newenvironment{varthm*}[1]{\trivlist\item[]{\bf #1.}\it}{\endtrivlist}

\renewcommand\geq{\geqslant}

\renewcommand\leq{\leqslant}

\newcommand\be{\begin{eqnarray*}}
\newcommand\ee{\end{eqnarray*}}

\newcommand\C{\mathbb C}

\newcommand\N{\mathbb N}

\renewcommand\P{\mathbb P}

\newcommand\calo{{\mathcal O}}
\newcommand\cali{{\mathcal I}}

\newcommand\call{{\mathcal L}}
\newcommand\calm{{\mathcal M}}
\newcommand\frakm{{\mathfrak m}}
\def\field{\C}
\newcommand\newop[2]{\def#1{\mathop{\rm #2}\nolimits}}
\newop\log{log}
\newop\ord{ord}
\newop\Gal{Gal}
\newop\SL{SL}
\newop\Bl{Bl}
\newop\mult{mult}
\newop\mass{mass}
\newop\div{div}
\newop\codim{codim}
\newop\sing{sing}
\newop\vdim{vdim}
\newop\edim{edim}
\newop\Ass{Ass}
\newop\size{size}
\newop\reg{reg}
\newop\satdeg{satdeg}
\newcommand\eqnref[1]{(\ref{#1})}

\newcommand\rru{\rule{1cm}{0cm}}

\def\keywordname{{\bfseries Keywords}}%
\def\keywords#1{\par\addvspace\medskipamount{\rightskip=0pt plus1cm
\def\and{\ifhmode\unskip\nobreak\fi\ $\cdot$
}\noindent\keywordname\enspace\ignorespaces#1\par}}
\def\subclassname{{\bfseries Mathematics Subject Classification
(2000)}\enspace}
\def\subclass#1{\par\addvspace\medskipamount{\rightskip=0pt plus1cm
\def\and{\ifhmode\unskip\nobreak\fi\ $\cdot$
}\noindent\subclassname\ignorespaces#1\par}}

\def\sys{\mathcal{L}}

\begin{document}

\author{M.~Dumnicki, T.~Szemberg\footnote{The second named author was partially supported
by NCN grant UMO-2011/01/B/ST1/04875}, H.~Tutaj-Gasi\'nska}
\title{A vanishing theorem and symbolic powers of planar point ideals}
\date{\today}
\maketitle
\thispagestyle{empty}

\begin{abstract}
   The purpose of this note is twofold. We present first a vanishing theorem
   (Theorem A) for families of linear series with base ideal being a fat points
   ideal. We apply then this result in order to give a partial proof
   of a conjecture raised by Bocci, Harbourne and Huneke concerning containment
   relations between ordinary and symbolic powers of planar point ideals (Theorem B).
\subclass{MSC 14C20 \and MSC 14F17 \and MSC 14H50 \and MSC 14J26 \and MSC 13A15 \and MSC 13F20}
\end{abstract}


\section{Introduction}
   One of the central problems in the theory of linear series
   is the study of linear systems of hypersurfaces in projective
   spaces with assigned base schemes. This problem is related
   to various other topics, for example to the polynomial
   interpolation, the Waring problem, the classification of
   defective higher secant varieties to mention a few in the
   realms of the algebraic geometry, to the problem
   of containment relations between ordinary and symbolic powers
   of ideals in commutative algebra and to problems
   in combinatorics \cite{Sul08}.

   Given a finite number $s$ of points $P_1,\dots,P_s$ in a projective
   space $\P^n$ and fixed integers $m_1,\dots,m_s$, one is interested
   in determining the dimension of the linear system
   $\call=\call_n(t;m_1,\dots,m_s)$ of hypersurfaces of a fixed degree $t$
   vanishing in the given set of points with prescribed multiplicities.
   The \emph{virtual dimension} of this space
   $$\vdim(\call)=\binom{n+t}{n}-\sum_{i=1}^s\binom{n+m_i-1}{n}-1$$
   arises by assuming that the conditions imposed by the underlying set
   of points are independent. The \emph{expected dimension}
   $$\edim(\call)=\max\left\{\vdim(\call),-1\right\}$$
   is just a modification of $\vdim(\call)$ taking into account
   the convention that the empty set has dimension $-1$. One always has
   \begin{equation}\label{eq:dim vdim}
      \dim(\call)\geq\edim(\call).
   \end{equation}
   If there is the equality in \eqnref{eq:dim vdim}, then we say
   that the linear system $\call$ is \emph{non-special}. Otherwise the
   system is \emph{special}.

   A subscheme $Z$ of $\P^n$ defined by an ideal of the form
   \begin{equation}\label{eq:fat points ideal}
      \cali_Z=\frakm_{P_1}^{m_1}\cap\dots\cap\frakm_{P_s}^{m_s},
   \end{equation}
   where $\frakm_{P}$ denotes the maximal ideal of a point $P\in\P^n$
   is called a \emph{fat points scheme} and the ideal $\cali_Z$ is called a \emph{fat points ideal}.
   It follows from the long cohomology sequence attached to the
   twisted structure sequence of $Z$
   $$0\to\cali_Z(t)\to\calo_{\P^n}(t)\to\calo_Z(t)\to 0$$
   that the system $\call$ is non-special exactly when the cohomology group
   \begin{equation}\label{eq:h1}
      H^1(\P^n,\cali_Z(t))
   \end{equation}
   vanishes. The system $\call$ is called \emph{$h^1$--regular} if
   it is non-special and effective, see Definition \ref{def:non spec and h1}.

   In the case of the projective plane, the non-speciality of linear
   series of type $\call$ is governed by the beautiful geometrical
   Segre-Harbourne-Gimigliano-Hirschowitz (SHGH for short) Conjecture,
   see e.g. \cite[Section 4]{Cil00} for precise statement and historical
   background.

   Since the SHGH Conjecture seems out of reach at present, it is of interest
   to provide other criteria ensuring the vanishing of the cohomology group in \eqnref{eq:h1}.
   Our first main result is the following vanishing theorem.
\begin{varthm*}{Theorem A}\label{thm:main 1}
   Let $P_1,\dots,P_s\in\P^2$ be $s\geq 4$ general planar points.
   Let $m_1\geq m_2\geq\dots\geq m_s\geq 1$ be fixed integers.
   If
   $$t\geq m_1+m_2\; \mbox{ and }\; \vdim(\call_2(t;m_1,\dots,m_s))\geq\frac12(3m_4^2-7m_4+2),$$
   then
   $$h^1(\P^2,\calo_{\P^2}(t)\otimes\frakm_{P_1}^{m_1}\otimes\dots\otimes\frakm_{P_s}^{m_s})=0$$
   i.e. the system $\call_2(t;m_1,\dots,m_s)$ is non-special.
\end{varthm*}
   Turning to the algebraic side of the story, let
   $\cali\subset\C[\P^n]=\C[x_0,\dots,x_n]$ be a homogeneous ideal.
   The $m$--th \emph{symbolic power} $\cali^{(m)}$ of $\cali$
   is defined as
   $$
   \cali^{(m)} = \field[\P^n] \cap \left( \bigcap_{\mathfrak{p} \in \Ass(\cali)} \cali^{m} \field[\P^n]_{\mathfrak{p}} \right),
   $$
   where the intersection is taken in the field of fractions of $\field[\P^n]$.
   If $\cali$ is a fat points ideal as in \eqnref{eq:fat points ideal},
   then the symbolic power is simply given by
   $$\cali^{(m)}=\bigcap_{i=1}^s\frakm_{P_i}^{m\cdot m_i}.$$
   There has been considerable interest in containment relations
   between usual and symbolic powers of homogeneous ideals over the
   last two decades. The most general results in this direction
   have been obtained with multiplier ideal techniques in characteristic
   zero by Ein, Lazarsfeld and Smith \cite{ELS01} and using tight
   closures in positive characteristic by Hochster and Huneke \cite{HoHu02}.
   Applying these results to a homogeneous ideal $\cali$ in the coordinate
   ring $\C[\P^n]$ of the projective space we obtain the following
   containment statement
\begin{equation}\label{eq:els containment}
   \cali^{(nr)}\subset \cali^r\;\mbox{ for all } r\geq 0.
\end{equation}
   There are examples showing that one cannot improve the power of the
   ideal $\cali$ on the right hand side of \eqnref{eq:els containment}.
   Nevertheless, it is natural to wonder to which extend this result can be improved
   for example under additional geometrical assumptions on the zero-locus of $\cali$.
   In particular, if $\cali$ is a fat points ideal, it is natural to wonder
   for which non-negative integers $m$, $r$ and $j$ there is the containment
   \begin{equation}\label{eq:containment problem}
      \cali^{(m)}\subset\calm^j\cali^r,
   \end{equation}
   where $\calm$ denotes the irrelevant ideal.
   Harbourne and Huneke suggested the following answer to that problem,
   \cite[Conjecture 2.1]{HaHu}. This assertion has guided our research leading to this note.
\begin{varthm*}{Conjecture}\label{conj:containment}
   Let $\cali$ be a fat points ideal in $\P^n$. Then
   $$\cali^{(nr)}\subset \calm^{r(n-1)}\cali^r$$
   for all $r\geq 1$.
\end{varthm*}
   This conjecture has been proved recently by
   Harbourne and Huneke for \emph{general} points
   in $\P^2$, \cite[Proposition 3.10]{HaHu}
   and by the first author for \emph{general}
   points in $\P^3$, \cite[Theorem 3]{Dum12}.
   In this note we extend these results to a large
   family of fat points ideals in $\P^2$. More specifically we show
\begin{varthm*}{Theorem B}\label{thm:main 2}
   Let $\cali=\frakm_{P_1}^{m_1}\cap\dots\cap\frakm_{P_s}^{m_s}$ be a fat
   points ideal supported on $s\geq 9$ general points in $\P^2$. If one of the
   following conditions holds
   \begin{itemize}
      \item[a)] at least $s-1$ among $m_i$'s are equal (\emph{almost homogeneous} case);
      \item[b)] $m_1\geq\dots\geq m_s\geq\frac{m_1}{2}$ (\emph{uniformly fat} case),
   \end{itemize}
   then the Conjecture above holds, i.e. there is the containment
   $$\cali^{(2r)}\subset\calm^r\cdot\cali^r$$
   for all $r\geq 1$.
\end{varthm*}
   We hope that more technical statements in Theorem \ref{marcin-witek},
   and Theorem \ref{zastosowanie} could be of independent interest when
   dealing with problems of similar flavor to those studied in this note.
\section{A reduction procedure}
   In this part we will be concerned with some operations on finite
   sequences of integers (which later on will be sequences of multiplicities).
   We begin by fixing some notation.
\begin{definition}
   Let $S=(a_1,a_2,\dots,a_k)$ be a sequence of non-negative integers.
   We call the number
   $$\size(S)=\sum_{i=1}^ka_i$$
   the \emph{size} of the sequence $S$.\\
   We say that a sequence $(a_1,\dots,a_k)$ \emph{is dominated}
   by a sequence $(b_1,\dots,b_{\ell})$ if
   $$k\leq \ell\;\mbox{ and }\; a_i\leq b_i\;\mbox{ for all }\; i=1,\dots,k.$$
\end{definition}
   The following reduction process was introduced in \cite{DumJar07}.
\begin{varthm*}{Reduction Algorithm with parameter $m$}\rm $\rule{.1cm}{0cm}$\\
   Let $S=(b_1,\dots,b_k,a_1,\dots,a_m)$ be a sequence of positive integers.
   \begin{itemize}
      \item[] $Z_m:=\left\{1,2,\dots,m\right\}$
      \item[] \textbf{for all} $k=m,m-1,\dots,1$ \textbf{do}
      \item[] \rru $z_k:=\max(Z_k)$
      \item[] \rru \textbf{if} ($a_k<m$ \textbf{and} $a_k \leq z_k$) \textbf{then} $r_k:=a_k$ \textbf{else} $r_k:=z_k$ \textbf{fi}
      \item[] \rru $c_k:=a_k-r_k$
      \item[] \rru $Z_{k-1}:=Z_k\setminus\left\{r_k\right\}$
      \item[] \rru \textbf{if} $Z_{k-1}=Z_k$ \textbf{then} stop \textbf{fi}
      \item[] \textbf{od}
   \end{itemize}
\end{varthm*}
   If the above process terminates by \textbf{stop}, then we say that
   the sequence $S$ \emph{is not $m$--reducible}. Otherwise $S$ \emph{is $m$--reducible}.
   The new sequence
   $$(b_1,\dots,b_k,c_1,\dots,c_m)$$
   is in that case called the \emph{$m$--reduction of} $S$.
   Note that
   \begin{equation}
      \size(b_1,\dots,b_k,c_1,\dots,c_m)=\size(b_1,\dots,b_k,a_1,\dots,a_m)-\frac{m(m+1)}{2}.
   \end{equation}
   We call the numbers $r_k$ appearing in the algorithm the \emph{reducers}.
   Thus at each stage of the reduction algorithm $Z_k$ is the set of available reducers.

   The following examples explain how the Algorithm works.
\begin{example}
   In the table below we present the input ($m$ and the indispensable sequence $a_1,\dots,a_m$),
   the reducers $r_1,\dots,r_m$, and the resulting sequence $c_1,\dots,c_m$, or the word \textbf{stop}, if the $m$--reduction fails.
   Algorithm steps are omitted.
$$
\begin{array}{cccccc}
&& m & a_1,\dots,a_m & r_1,\dots,r_m & c_1,\dots,c_m \\
A) && 3 & 5,5,5 & 1,2,3 & 4,3,2 \\
B) && 4 & 5,5,3,1 & 2,4,3,1 & 3,1,0,0 \\
C) && 3 & 4,1,3 & 2,1,3 & 2,0,0 \\
D) && 3 & 4,2,2 & -,2,2 & \textbf{stop} \, (Z_1 = Z_2)
\end{array}
$$
\end{example}
   The next example is more involved. We trace the algorithm steps.
\begin{example}\label{ex:with steps}
   In this example we will perform a sequence of reductions in such a way that the result of $j$--th reduction will be
   the input for $(j+1)$--st reduction. Observe that the Reduction Algorithm works on a sequence of positive integers, so we must
   shorten the sequence of integers by deleting zeroes at the end. We will start with $(a_1,\dots,a_{10})=(1,2,\dots,10)$ and use $m_j$--reductions,
   $m_1=m_2=m_3=4$, $m_4=\ldots=m_7=3$. We will gather all intermediate steps in the table, together with the used reducers.
   The table shows what happens to each sequence member.
   For example the seventh element, $a_7=7$, has been altered to $6$, then to $5$, $4$, $1$ and finally to $0$.
$$
\begin{array}{cccccccccc|c}
a_1 & a_2 & a_3 & a_4 & a_5 & a_6 & a_7 & a_8 & a_9 & a_{10} & m_j \\
1 & 2 & 3 & 4 & 5 & 6 & 7 & 8 & 9 & 10 & \\
  &   &   &   &   &   &-1 &-2 &-3 &-4 & 4 \\
1 & 2 & 3 & 4 & 5 & 6 & 6 & 6 & 6 & 6 & \\
  &   &   &   &   &   &-1 &-2 &-3 &-4 & 4 \\
1 & 2 & 3 & 4 & 5 & 6 & 5 & 4 & 3 & 2 & \\
  &   &   &   &   &   &-1 &-4 &-3 &-2 & 4 \\
1 & 2 & 3 & 4 & 5 & 6 & 4 & 0 & 0 & 0 & \\
  &   &   &   &-1 &-2 &-3 &   &   &   & 3 \\
1 & 2 & 3 & 4 & 4 & 4 & 1 &   &   &   & \\
  &   &   &   &-2 &-3 &-1 &   &   &   & 3 \\
1 & 2 & 3 & 4 & 2 & 1 & 0 &   &   &   & \\
  &   &   &-3 &-2 &-1 &   &   &   &   & 3 \\
1 & 2 & 3 & 1 & 0 & 0 &   &   &   &   & \\
  &-2 &-3 &-1 &   &   &   &   &   &   & 3 \\
1 & 0 & 0 & 0 &   &   &   &   &   &   & \\
\end{array}
$$
   This example is continued in Remark \ref{rem:ex with continued}
\end{example}
   The reader may check his understanding of the Reduction process verifying the claims
   in the following two examples.
\begin{example}
   The sequence $(2,3,4,5,2)$ is $4$-reducible, its $4$-reduction equals to $(2,2,1,1)$.
\end{example}
\begin{example}
   The sequence
   $(1,2,3,3)$ is not $4$-reducible (we get $r_4=r_3=3$).\\
   The sequence $(3,2,4,2)$ is not
   $4$-reducible (we get $r_2=r_4=2$).
\end{example}
   The following Lemma gives first insights into what happens to a sequence
   of integers under the reduction algorithm.
\begin{lemma}\label{lem:reddown}
   Let $S=(b_1,\dots,b_{\ell-1},b_{\ell},\dots,b_r)$ be a given sequence of integers
   to which the Reduction Algorithm has been successfully applied. Assume that the
   original sequence $S$ has been altered on positions with labels $\ell-1$ and $\ell$
   and that the resulting sequence
   has the form $(\dots,a_{\ell-1},a_{\ell},\dots)$.
   Then either $a_{\ell} = 0$ or
   $$b_{\ell-1}-b_{\ell}< a_{\ell-1}-a_{\ell}.$$
\end{lemma}
\proof
   In the reduction process either $b_\ell$ is reduced to zero, or the number $r_{\ell}$, which we substract from $b_{\ell}$,
   is chosen as a maximum $z_{\ell}$
   of $Z_{\ell}$, which gives $z_{\ell-1} < z_{\ell}$. Since always $r_{\ell-1} \leq z_{\ell-1}$, the number $r_{\ell}$ which we
   subtract from $b_{\ell}$ is always greater
   than the number $r_{\ell-1}$ which we subtract from $b_{\ell -1}$.
   Thus
   $$a_{\ell-1}-a_{\ell}=b_{\ell-1}-r_{\ell-1}-b_{\ell}+r_{\ell}=b_{\ell-1}-b_{\ell}+r_{\ell}-r_{\ell-1}>b_{\ell-1}-b_{\ell}.$$
\endproof
\begin{corollary}\label{cor:reddown}
   Assume that $(a_1,\dots,a_r)$ is a sequence of integers
   obtained by a sequence of successful reductions applied consecutively to the
   sequence $(1,2,\dots,r)$ (some of $a_j$'s might be zero). Then, for
   every $k=1,\dots,r$, either $a_k=k$ or $a_k \geq a_{k+1} \geq \ldots \geq a_r$. In particular,
   if $a_k=0$ then $a_{k+1}=0$. \\
   If moreover the $k$--th position of the sequence has been altered by at least two reductions, then
   either $a_{k-1} > a_{k}$ or $a_{k-1}=0$.
\end{corollary}
\proof
   Let us assume that $a_k < k$. It follows that the element on the $k$--th
   position has been altered in the course of at least one reduction.
   Since reductions work from the
   right to the left, also each element on the $\ell$-th position, for $\ell > k$, has
   been altered. Therefore it is enough to prove that if $a_k$ has
   been altered by a reduction, then $a_k \geq a_{k+1}$. This follows immediately from Lemma \ref{lem:reddown}.
   The last property follows straightforward.
\endproof
   The next Lemma describes numerical properties
   of sequences which are not $m$--reducible. It is going
   to be important in the proof of Theorem \ref{marcin-kryterium}.
\begin{lemma}[A consequence of non $m$--reducibility]\label{lem:redfails}
   Let $(b_1,\dots,b_r)$ be a sequence of integers
   which is not $m$--reducible. Then one of the following two conditions is satisfied:
   \begin{itemize}
      \item[1).] $r < m$ (the sequence is too short) or;
      \item[2).] there exist $k,\ell$ such that $r-m+1 \leq k <\ell \leq r$ and $b_{k} \leq b_{\ell}$, $b_{k} < m$ (the sequence has a too flat tail).
   \end{itemize}
\end{lemma}
\begin{proof}
   The first property is obvious, so we assume that $r \geq m$.
   We adjust the notation to match the Reduction Algorithm and write
   $$(b_1,\dots,b_r)=(b_1,\dots,b_{r-m},a_1,\dots,a_m)$$
   with $a_i=b_{r-m+i}$ for $i=1,\dots,m$.
   Let $r_{k'+1},\dots,r_m$ be the sequence of reducers and let
   $c_{k'+1},\dots,c_m$ be the sequence resulting
   from the Reduction Algorithm right before the moment it stopped
   (i.e. we assume that it stopped for $k'$). The algorithm stops
   if $Z_{k'-1}=Z_{k'}$. This means that $r_{k'}$ was not an
   element of $Z_{k'}$, so that in particular this reducer has
   been used in a previous step, say with index $\ell'>k'$.
   Moreover this implies that $r_{k'}\neq z_{k'}$, which
   going back one line in the algorithm implies that
   $r_{k'}=a_{k'}$ and consequently $a_{k'}<m$.
   Thus the claim follows with $k=r-m+k'$ and $\ell=r-m+\ell'$.
\end{proof}

\section{A vanishing theorem}
   We begin by recalling from \cite{DumJar07} that on $\P^n$ one can consider slightly more general linear
   series than of the form $\call_n(t;m_1,\dots,m_s)$. Specifically, let $\call_{G,n}(m_1,\dots,m_s)$
   be the vector space spanned by all monomials in the ideal generated by monomials
   with exponents in a fixed set $G\subset \N^n$. We suppress the index $n$ if the dimension
   of the ambient space is understood. Thus for example for $n=2$ we have
   $$\call(t;m_1,\dots,m_s)=\call_D(m_1,\dots,m_s),$$
   where $D = \{ (x,y) \in \N^2 : x+y \leq t \}$. We extend the notions
   of non-speciality and $h^1$--regularity to linear systems of that kind, see \cite[Definition 5]{DumJar07}.
\begin{definition}[Non-speciality and $h^1$--regularity]\label{def:non spec and h1}
   We say that the linear system $\call_G(m_1,\dots,m_s)$ is non-special
   if its dimension agrees with the virtual dimension.\\
   We say that the system $\call_G(m_1,\dots,m_s)$ is $h^1$--regular, if
   it is non-special and effective.
\end{definition}
\begin{theorem}\label{marcin-witek}
   Fix positive integers $d,m_1,\dots,m_s$ and let $S_1=(1,2,\dots,d,d+1)$. Assume inductively that
   $S_j$ is $m_j$-reducible with the reduction equal to $S_{j+1}$, $j=1,\dots,s$. If $\size(S_{s+1})>0$ is positive, then
   the linear system $\call(d;m_1,\dots,m_s)$ is $h^1$-regular, assuming that the multiplicities
   are imposed in \emph{general} points.
\end{theorem}
\begin{proof}
   With notation recalled above, the result follows from \cite{DumJar07}.
   By \cite[Definition 11]{DumJar07} we can write $D=S_1$. Since $S_j$ is $m_j$-reducible to $S_{j+1}$, by
   \cite[Corollary 19]{DumJar07} we know that if $\call_{S_{j+1}}(m_{j+1},\dots,m_s)$ is non-special, then
   $\call_{S_{j}}(m_j,m_{j+1},\dots,m_s)$ is also non-special. We start with the trivial case of $\call_{S_{s+1}}()$
   (with no multiplicities imposed), which is trivially non-special, and go back inductively to obtain non-speciality
   of $\call_{S_{1}}(m_1,\dots,m_s)$. Observe that non-emptiness of $S_{s+1}$ gives non-emptiness of $\call_{S_{1}}(m_1,\dots,m_s)$.
\end{proof}
   The above Theorem whereas powerful in very concrete questions, is not easily applicable
   under general assumptions. We present below its modification which is well suited for
   the proof of Theorem A.
\begin{theorem}\label{marcin-kryterium}
   Let $s,d,m_1,\dots,m_s$ be positive integers, with $s\geq 4$.
   Assume that the multiplicities are ordered
   $m_1 \geq m_2 \geq m_3\geq \ldots \geq m_s$.
   If $d \geq m_1+m_2$ and
   $$\binom{d+2}{2} - \sum_{j=1}^{s} \binom{m_j+1}{2} \geq (2m_4-1)(m_4-1) + 1 - \frac{m_4(m_4+1)}{2}
      =\frac32m_4^2-\frac72m_4+2,$$
   then $\sys(d;m_1,\dots,m_s)$ is non-special and $h^1$-regular whenever the left-hand side is non-negative,
   in particular for $m_4\geq 2$.
\end{theorem}

\begin{proof}
   The key idea of the proof is to apply the Reduction Algorithm $s$ times with $m$ equal to
   $$m_1,m_2,m_3,m_5,\dots,m_s,m_4.$$

   We begin by showing that the first three reductions are always possible.
   Indeed,
   the sequence
   $$(1,2,...,d+1)\; \mbox{ is }\; m_1-\mbox{reducible}$$
   since $m_1\leq d$. After $m_1$--reduction we obtain the sequence
   $$(1,2,...,d+1-m_1,d+1-m_1,d+1-m_1,...,d+1-m_1)$$
   in which the terms $d+1-m_1$ appear $m_1+1$ times.

   The $m_2$--reduction works for this sequence because
   we have inequalities $m_2\leq m_1$ and $d\geq m_1+m_2$.
   We obtain the sequence
   \begin{equation}\label{eq:red seq after m2}
      (1,2,...,d+1-m_1,...,d+3-m_1-m_2,d+2-m_1-m_2,d+1-m_1-m_2)
   \end{equation}
   where $d+1-m_1$ appears at  least once.

   Again, since $d\geq m_1+m_2$, the last element in the reduced sequence \eqnref{eq:red seq after m2}
   is at least $1$. Since $m_3\leq m_2$
   the $m_3$--reduction is possible on the sequence \eqnref{eq:red seq after m2}.
   Indeed, no situation in Lemma \ref{lem:redfails} can occur.
   After this reduction step, some elements in the sequence could be $0$ and
   they are removed from the tail.

   If $m_4=1$, then also $m_5=\dots=m_s=1$ and we are done because the
   Reduction Algorithm always works for multiplicity $1$. So we may assume
   from now on that $m_4\geq 2$.

   Let us assume that for some $j>3$ the $m_j$--reduction failed
   (and that all previous reductions were possible). Let
   $(a_1,..,a_r)$ denote the state of the sequence before the $m_j$--reduction.
   By Lemma \ref{lem:redfails} we have two possibilities.

\textbf{Case 1.} The sequence $(a_1,\dots,a_r)$ is too short to be reduced i.e. $r<m_j\leq m_4$.
   Thus the sequence $(a_1,\dots,a_r)$ is dominated by $(1,2,\dots,m_4-1)$ with
   $\size(1,2,\dots,m_4-1)=\frac{(m_4-1)m_4}{2}$. A contradiction.

\textbf{Case 2.} Assume that the sequence is $(a_1,\dots,a_k,\dots,a_{\ell},\dots,a_r)$, with
   \begin{equation}\label{eq:case 3}
   a_k \leq a_{\ell}, \; a_{k} < m_j \; \mbox{ and }\; r-m_j+1\leq k.
   \end{equation}

   We have now again two possibilities:

\textbf{Subcase 2.1.} Assume that $a_k=k$. Then $k=a_k<m_j \leq m_4$,
   and taking into account \eqnref{eq:case 3} we have $r \leq 2m_4-2$. Thus the
   sequence $(a_1,\dots,a_r)$ is dominated by $(1,\dots,2m_4-2)$ and
   \begin{equation}\label{eq:case 3a}
   \size(1,2,...,2m_4-2)=(2m_4-1)(m_4-1).
   \end{equation}
   A contradiction.

\textbf{Subcase 2.2.} Assume that $a_k<k$. This is the most tricky situation.
   By Corollary \ref{cor:reddown}
   $a_k \geq a_{k+1} \geq a_{k+2} \geq \ldots \geq a_{\ell} \geq \ldots \geq a_r$.
   Since $a_k \leq a_{\ell}$, we have $a_k=a_{k+1}=\ldots=a_{\ell}$.

   Assume for a moment that
   \begin{equation}\label{eq:ass l 2m4}
      \ell \geq 2m_4.
   \end{equation}
   We know that the $\ell$-th entry had been reduced at
   least once. But if it had been reduced twice or more, Corollary \ref{cor:reddown} would imply
   that $a_{\ell-1}> a_{\ell}$, which is impossible. So it was reduced exactly once.

   Now we claim the following:
\\
   \textit{Claim 1. ${\ell}$-th entry was \textbf{not} reduced by the first three reductions.}
\\
   Indeed, after the $m_1$--reduction it would become $a_{\ell} = d+1-m_1$.
   But, we have $a_{\ell}<m_j\leq m_2$, which contradicts our assumption that $d \geq m_1+m_2$.

   Reductions with $m_2$ and $m_3$ are also excluded because these reductions
   work only on elements reduced in the previous step (i.e. either by the $m_1$-reduction
   or by the $m_2$-reduction). But then $a_{\ell}$ would be reduced at least twice.
   This contradiction justifies Claim 1.

   Thus,  the $\ell$-th entry was reduced by some $m_i$--reduction with $i > 3.$
   Then the assumption \eqnref{eq:ass l 2m4} gives $a_{\ell} \geq 2m_4-m_j\geq m_4$.
   This is a contradiction again, as the assumptions of Case 2 imply that $m_4>a_{\ell}$.

   This way we have proved that  $\ell \leq 2m_4-1$. Hence $k-1 \leq 2m_4-3$ and our sequence is dominated
   by $(1,2,\dots,2m_4-3,m_4-1,\dots,m_4-1)$ (with $m_4-1$ appearing at most $m_4$ times). In fact
   it appears only twice, this is our next Claim.
\\
   \textit{Claim 2. In the sequence $(1,2,\dots,2m_4-3,m_4-1,\dots,m_4-1)$,
   the $m_4-1$ term appears at most two times.}
\\
   Assume that $a_r \neq 0$ (otherwise we operate on a shorter sequence). It follows that $a_j \neq
   0$ for $j=1,\dots,r$. The sequence $(a_1,\dots,a_k,\dots,a_\ell,\dots,a_r)$ results from a
   $m_{j-1}$--reduction of the sequence $(b_1,\dots,b_k,\dots,b_\ell,\dots,b_r,b_{r+1},\dots)$
   for some $j\geq 6$.

   We know that the $\ell$-th element was reduced exactly once
   (otherwise we would have $a_k>a_\ell$), hence $b_\ell=\ell$. Let
   $\tilde{k}:=r-\ell$. Observe that $b_r$ has been reduced to a
   non-zero $a_r$. This is possible only when the reducer for $b_r$
   has been chosen as a maximum of non-used reducers. Hence either
   $m_{j-1}$, as a maximal reducer, has been chosen before, or it is chosen in the $(m_{j-1})$--reduction
   to reduce $b_r$. It follows that the maximal reducer
   for $b_{r-1}$ is $m_{j-1}-1$. Inductively, the maximal reducer
   for $b_{\ell}=b_{r-\tilde{k}}$ is $m_{j-1}-\tilde{k}$, but this implies
   that $b_{\ell} \leq a_{\ell}+(m_{j-1}-\tilde{k})$, so we have the
   following sequence of inequalities:
   $$\ell = b_{\ell} \leq a_{\ell}+m_{j-1}-\tilde{k} \leq m_4-1+m_4-\tilde{k}=2m_4-1+\ell-r,$$
   which gives $r \leq 2m_4-1$. This proves Claim 2.

   That claim, together with the previous bound for $(a_1,\dots,a_r)$, gives that $(a_1,\dots,a_r)$ is
   dominated by $(1,2,\dots,2m_4-3,m_4-1,m_4-1)$ whose size is $(2m_4-1)(m_4-1)$.

\bigskip

   Concluding, we see that if the size of some sequence obtained during reducing is at least
   $(2m_4-1)(m_4-1) + 1$ then it is $p$--reducible for each $p \leq m_4$.

   The size of our sequence $(1,\dots,d+1)$, at the beginning, is $(d+1)(d+2)/2$. After the $j$-th
   reduction it is (remember that we reduce using $m_1$, $m_2$, $m_3$, $m_5, \dots$ in this order)
   $$\zeta(j) := \frac{(d+1)(d+2)}{2} - \frac{m_1(m_1+1)}{2} - \frac{m_2(m_2+1)}{2} - \frac{m_3(m_3+1)}{2} - \sum_{k=5}^{j+1} \frac{m_k(m_k+1)}{2}.$$

   As long as $\zeta(j) \geq (2m_4-1)(m_4-1)+1$ holds, the reduction is possible.
   Note that $\zeta(s-1) = \vdim \sys(d;m_1,\dots,m_s) + 1 + m_4(m_4+1)/2$.
   If $\zeta(s-1)$ is greater or equal to $(2m_4-1)(m_4-1)+1$, then all reductions are possible.
\end{proof}
   An immediate corollary useful for the proof of Theorem B is the following.
\begin{corollary}
   Keeping the notation from Theorem \ref{marcin-kryterium}, if the system in the Theorem is $h^1$--regular, then
   the Castelnuovo-Mumford regularity of the ideal $I=\bigcap_{j=1}^{s} \mathfrak{m}_{p_j}^{m_j}$ is less or equal $d+1$.
\end{corollary}
\begin{remark}\label{rem:ex with continued}
   Continuing Example \ref{ex:with steps}, we observe that Theorem \ref{marcin-witek} implies that
   the system $\call(9;4,4,4,3,3,3,3)$ is $h^1$-regular.
   With Theorem \ref{marcin-kryterium} we can show only that
   $\call(9;4,4,4,3,3)$ is $h^1$-regular.
   In fact, from the proof of Theorem \ref{marcin-kryterium} we know that performing
   three $4$-reductions and two $3$-reductions on $(1,\dots,10)$
   is possible, even without performing them as in Example \ref{ex:with steps}.
   This example shows that in some situations Theorem \ref{marcin-witek}
   is stronger than Theorem \ref{marcin-kryterium}. However Theorem \ref{marcin-witek} is simply not
   so useful to handle general situations.
\end{remark}

\section{The containment results}
   The following result of Harbourne and Huneke
   \cite[Proposition 3.10]{HaHu} has motivated this part of the article.
\begin{theorem}[Harbourne, Huneke]\label{thm:hahu}
   Let $\cali=\frakm_{P_1}\cap\dots\frakm_{P_s}$ be an ideal
   supported on $s$ general points in $\P^2$. Then
   $$\cali^{(2r)}\subset \calm^r\cali^r$$
   for all $r\geq 1$.
\end{theorem}
   Our arguments in this part rely on the following fact modelled
   on a result by Harbourne and Huneke,
   \cite[Lemma 2.3]{HaHu}.
\begin{lemma}\label{lemathahu}
   Let $\cali$ be a homogeneous ideal in $(n+1)$ variables
   with $0$--dimensional support. Assume that
   for some non-negative integers $r$ and $k$
   \begin{equation}\label{eq:bound on alpha}
      \alpha(\cali^{(q)})\geq r\cdot\reg(\cali)+k.
   \end{equation}
   Then
   $$\cali^{(q)}\subset \calm^k\cdot \cali^r.$$
\end{lemma}
\proof
   Note that since $\alpha(\cali^{(q)})\leq q\cdot \alpha(\cali)$,
   we get from \eqnref{eq:bound on alpha} that $r\leq q$, which
   in particular shows the inclusion $\cali^{(q)}\subset \cali^{(r)}$.

   For the Castelnuovo-Mumford regularity in our situation,
   we have by \cite[Theorem 1.1]{GGP95}
   $$\satdeg(\cali^r)\leq\reg(\cali^r)\leq r\cdot\reg(\cali),$$
   so that
   $$\left(\cali^{(r)}\right)_t= \left(\cali^r\right)_t\;\mbox{ for }\; t\geq r\cdot\reg(\cali).$$
   Hence
   \begin{equation}\label{eq:q in r}
      \left(\cali^{(q)}\right)_t\subset \left(\cali^r\right)_t\;\mbox{ for }\; t\geq r\cdot\reg(\cali).
   \end{equation}
   Let $h_1,\dots,h_{\ell}$ be minimal degree (i.e. $\leq r\cdot t$) generators
   of $\cali^r$. For $f\in \left(\cali^{(q)}\right)_t$ we have
   $f=0$ if $t<r\cdot\reg(\cali)+k$ by the assumption \eqnref{eq:bound on alpha}.
   On the other hand, for $t\geq r\cdot\reg(\cali)+k$
   by \eqnref{eq:q in r} there exist homogeneous polynomials $f_1,\dots,f_{\ell}$
   such that $f=\sum_{i=1}^{\ell}f_i\cdot h_i$. It follows that
   $$\deg(f_i)\geq\deg(f)-\deg(h_i)\geq k$$
   for all $i=1,\dots,\ell$, which provides the desired result.
\endproof
   As an immediate consequence of the above Lemma, we obtain the following useful criterion.
\begin{corollary}\label{cor:hahu}
   Let $\cali$ be a fat points ideal in $\P^2$.
   Assume that
   $$\alpha(\cali^{(2r)})\geq r\cdot(\reg(\cali)+1).$$
   Then
   $$\cali^{(2r)}\subset \calm^r\cali^r.$$
\end{corollary}
   Before proceeding, it is convenient
   to introduce the following function
   $$\rho(m):=\left\{\begin{array}{ccc}
      0 & \mbox{ if } & m=1\\
      (3m-1)(m-2) & \mbox{ if }  & m\geq 2
      \end{array}\right..
      $$
   The crucial point in the proof of Theorem B is the following criterion,
   which follows from the $h^1$--regularity statement in Theorem \ref{marcin-kryterium}.
\begin{theorem}\label{zastosowanie}
   Let $m_1\geq m_2 \geq ...\geq m_s$ be positive integers and consider the fat points ideal
   $\cali=\mathfrak{m}_{P_1}^{m_1}\cap\mathfrak{m}_{P_2}^{m_2}\cap...\cap\mathfrak{m}_{P_s}^{m_s}$ of
   $s\geq 9$ general points. If there exists an integer $d$ such that
\begin{equation}
\label{eqreg} d(d+3)\geq \sum_{i=1}^s m_i(m_i+1)+\rho(m_4)
\end{equation}
   \centerline{and}
\begin{equation}\label{eq:m12}
   d\geq m_1+m_2
\end{equation}
   \centerline{and}
\begin{equation}
\label{eqalpha}
d+2\leq \max\left\{\frac{2}{\sqrt{s+1}}\sum_{i=1}^s m_i, \quad
m_1+m_2+m_3+m_4, \quad 2m_1\right\},
\end{equation}
   then $\cali^{(2r)}\subset \calm^r\cali^r$ for all $r\geq 1$.
\end{theorem}

\begin{proof}
  It follows from Theorem \ref{marcin-kryterium} that \eqnref{eqreg} and \eqnref{eq:m12}
  imply that $\reg(\cali)\leq d+1$, so that $\cali$ is generated
  in degree $d+1$. In order to apply Corollary \ref{cor:hahu} we
  have to check that the inequality $\alpha(\cali^{(2r)})\geq r(d+2)$ holds.

  By assumption \eqnref{eqalpha} we need to consider three cases.\\
  Assume first that
\begin{equation}
\label{eqnagata}
   d+2 \leq \frac{2}{\sqrt{s+1}}\sum_{i=1}^s m_i.
\end{equation}
   A lower bound on Seshadri constant in $s\geq 9$ general points in $\mathbb{P}^2$, see e.g. \cite[Theorem 1(a)]{Xu94}
   \begin{equation}\label{eq:sesh bound}
      \varepsilon({\cal O}_{\mathbb{P}^2}(1),P_1,...,P_s)\geq \frac{1}{\sqrt{s+1}}
   \end{equation}
   combined with \eqnref{eqnagata} implies that $(\cali^{(2r)})_{r(d+2)}$ is empty
   so that $\alpha(\cali^{(2r)})> r(d+2)$ in that case.

   In the second case we assume
   $$d+2 \leq m_1+m_2+m_3+m_4.$$
   Then the standard Cremona transformation applied to the system
   $\call(r(d+2)-1;2m_1,...,2m_s)$ gives
   $$\call(2r(d+2-m_1-m_2-m_3-1); $$$$r(d+2-2m_2-2m_3)-1, r(d+2-2m_1-2m_3)-1, r(d+2-2m_1-2m_2)-1,2m_4,...,2m_s)$$
   which is obviously empty since its degree is less than the fourth multiplicity.

   Finally, from the assumption $d+2\leq 2m_1$ it follows immediately that
   $$(\cali^{(2r)})_{r(d+2)-1}=0$$
   and we are done.
\end{proof}
\begin{remark}\rm
   Computer experiments suggest that for nearly all sequences of multiplicities the assumptions of Theorem 4.4 are fulfilled.
   It took even some time (and without computer aid it would not be so easy) to find some examples, where Theorem \ref{zastosowanie}
   does not prove the desired inclusion. Such sequences of multiplicities are for example
   $$(8^9,1^{103});\;\;
     (9^{11},1^{80});\;\;
     (20^{12},2^{90});\;\;
     (30^{11},3^{130});\;\;
     (60^{11},5^{224});\;\;
     (130^{12},12^{101}),$$
   where the notation $a^b$ means that $a$ appears in the sequence $b$ times.
   All these examples are of similar nature, namely few points with high multiplicity, and a long tail of low multiplicities.
   It would be interesting to find at least bounds on the initial degree of ideals associated to them
   in the spirit of the proof of Theorem \ref{zastosowanie}.
\end{remark}
   In the sequel, we will frequently use the following purely numerical observation.
\begin{lemma}\label{lem:gwiazdka}
   Let $R,D\geq 0$ be real numbers. If $R^2-3R\geq D$, then there exists an integer $d$
   such that
   \begin{itemize}
    \item[a)] $d(d+3)\geq D$ and,
    \item[b)] $d+2\leq R$.
   \end{itemize}
\end{lemma}
\proof
   Let $L_0:=\frac{\sqrt{9+4\cdot D}-3}{2}$ be the greater root of the equation $L(L+3)-D=0$.
   Obviously for any $d\geq L_0$ condition a) in the Lemma holds.
   In order to check condition b), we show that there exists an integer $d$ in the interval $[L_0,R-2]$.
   This is immediate, once we show that the length of this interval is at least $1$, i.e.
   $R-2-L_0\geq 1$ holds. But this follows immediately from the assumption in the Lemma.
\endproof
   The proof of Theorem B will be split in two cases.
\subsection{Proof of Theorem B a)}
   We begin with the overview of the structure of the proof of
   part a) of Theorem B. It follows from
   \begin{center}
   \begin{tabular}{rcl}
      Theorem \ref{thm:hahu} & for & $m_1=\dots=m_s=1$,\\
      Proposition \ref{drugie} & for & $m_1\geq 2$ and $m_2=\dots=m_s=1$,\\
      Theorem \ref{result2} & for & all other cases.
   \end{tabular}
   \end{center}
   It is convenient in the almost homogeneous case we study here
   to change a little bit the notation and begin the numbering of points by $0$
   rather than $1$, the point $P_0$ being the point with distinguished multiplicity.
   This convention simplifies the notation below. We hope that this will cause no confusion
   and the reader will have no difficulties to modify Theorem B a) accordingly.

\begin{proposition}\label{drugie}
   Let $P_0,...,P_s$ be general points on $\mathbb{P}^2$, with $s+1\geq 9$. Let
   $\cali=\mathfrak{m}_{P_0}^{m_0}\cap\mathfrak{m}_{P_1}\cap...\cap\mathfrak{m}_{P_s}$, $m_0 \geq 2$.
   Then $\cali^{(2r)}\subset \calm^r\cali^r$ for all $r\geq 1$.
\end{proposition}
\proof

   Assume that the regularity of $\cali$ is $t$. The points being general impose independent
   conditions, so that we have
   \begin{equation}\label{eq:tsk}
      \binom{t}{2}< s+\binom{m_0+1}{2}.
   \end{equation}
   We claim that $\alpha(\cali^{(2r)})\geq r(t+1)$.
   Taking this for granted for the moment, the assertion follows from Lemma
   \ref{lemathahu}.\\
   Turning to the claim, the lower bound on Seshadri constants
   \eqnref{eq:sesh bound} implies that
   $$\alpha(\cali^{(2r)})\geq\frac{2r(m_0+s)}{\sqrt{s+2}}.$$
   Since the multiplicity of an element in $\cali^{(2r)}$
   in the distinguished point is $2rm_0$ we have additionally
   $$\alpha(\cali^{(2r)})\geq 2rm_0.$$
   Thus, dividing by $r$, it suffices to prove that the following
   \begin{equation}\label{eq:both}
   2\frac{m_0+s}{\sqrt{s+2}} \geq t+1\; \text{ or }\; 2m_0 \geq t+1
   \end{equation}
   holds for all $t$ satisfying \eqnref{eq:tsk}.

   Assume that \eqnref{eq:both} does not hold and plug
   $$t = 2\frac{m_0+s}{\sqrt{s+2}} - 1$$ into \eqnref{eq:tsk}.
   After a small computation we obtain the inequality
   $$\frac{4m_0^2+8m_0s+4s^2}{s+2} +2 \leq 2s+m_0^2+m_0+\frac{6m_0}{\sqrt{s+2}}+\frac{6s}{\sqrt{s+2}}.$$

   Let us assume that $s \geq 34$, so that in particular $\frac{6}{\sqrt{s+2}}\leq 1$. Then
   $$\frac{4m_0^2+8m_0s+4s^2}{s+2} \leq 3s+m_0^2+2m_0,$$
   hence
   $$2m_0^2+6m_0s+s^2 \leq 6s+sm_0^2+4m_0.$$
   But for $m_0^2 \leq s$ we have
   $$2m_0^2 > 4m_0, \quad 6m_0s > 6s \;\mbox{ and }\; s^2 \geq sm_0^2,$$
   which gives a contradiction.

   Now we do the same for $t=2m_0-1$ obtaining
   $$3m_0^2-7m_0+2 \leq 2s$$ and observe that for $s \leq m_0^2$ and $m_0 \geq 7$ this is absurd.

   Thus we are left with a finite number of pairs $(m_0,s)$ to check, namely
   $m_0^2 \leq s \leq 33$ and $s \leq m_0^2 \leq 36$. In each case we directly compute maximal possible $t$ satisfying \eqnref{eq:tsk}
   and check that it fits to \eqnref{eq:both}.
\endproof

   In order to prove Theorem \ref{result2} below, we need first the following Lemma.
   We abbreviate $\Sigma=sm+m_0$ and $Q=sm^2+m_0^2$.
\begin{lemma}\label{hopefullylast}
   Let $m_0\geq 1$, $m\geq 2$ and $s\geq 8$ be integers.
   \begin{itemize}
      \item[a)] if $m_0\geq \frac{\Sigma}{\sqrt{s+2}}$, then $4m_0^2\geq Q+\Sigma+3m^2+6m_0$,
      \item[b)] if $m_0\leq \frac{\Sigma}{\sqrt{s+2}}$, then $\frac{4}{s+2}\Sigma^2\geq m_0^2+(s+3)m^2+3\Sigma$.
   \end{itemize}
\end{lemma}
\proof
   \textbf{Part a)} An elementary computation shows that the
   assumption in Part a) together with the following inequality
   \begin{equation}\label{eq:N1}
      3(m_0+sm)^2\geq (s+2)(s+3)m^2+7(s+2)m_0+s(s+2)m
   \end{equation}
   implies the assertion. So that it is enough to show
   \eqnref{eq:N1}.
   It follows since for $s\geq 8$ and $m\geq 2$ the following inequalities are satisfied:
   $$2s^2m^2\geq (s+2)(s+3)m^2,$$
   $$s^2m^2\geq s(s+2)m,$$
   $$3m_0^2+6smm_0\geq 7(s+2)m_0.$$
   Adding them gives \eqnref{eq:N1}.
\\
   \textbf{Part b)}
   Adding
   $$\Sigma^2\geq(s+2)m_0^2,$$
   which follows from the assumption in this part,
   to the following inequality
   \begin{equation}\label{nier}
      3(m_0+sm)^2\geq (s+2)(s+3)m^2+3(s+2)(m_0+sm)
   \end{equation}
   one obtains the inequality claimed in b). So it suffices to
   prove \eqnref{nier}.
   For all $s\geq 8$ and $m\geq 2$ we have
   $$3m_0+6sm\geq 3s+6,$$
   so that it is enough to check that
   \begin{equation}\label{eq:nier2}
      3s^2m\geq (s+2)(s+3)m+3(s+2)s.
   \end{equation}
   For $s\geq 8$ we have
   $$2s^2\geq (s+2)(s+3)$$ so \eqnref{eq:nier2}
   reduces to
   \begin{equation}\label{eq:nier3}
     ms^2\geq 3s^2+6s.
   \end{equation}
   This inequality is satisfied for all $m\geq 4$.\\

   For $m=2$ or $m=3$ we show \eqnref{eq:nier2} slightly
   differently. Namely, we bound the summands on the right hand
   side in the following way
   $$\frac54s^2\geq(s+2)(s+3)\;\mbox{ and }\; \frac74s^2m\geq
   3(s+2),$$
   which holds for all $s\geq 22$.

   In all remaining, finitely many cases, i.e. $m=2$ or $m=3$ and $8\leq s\leq 21$
   we check the inequality in Part b) directly by hand.
\endproof
   We are now in the position to justify part a) of Theorem B.
\begin{theorem}\label{result2}
   Let $m_0\geq 1$ and $m\geq 2$ be integers. Consider the ideal
   $\cali=\mathfrak{m}_{P_0}^{m_0}\cap\mathfrak{m}_{P_1}^{m}\cap...\cap\mathfrak{m}_{P_{s}}^{m}$ of
   $s+1\geq 9$ general (fat) points. Then $\cali^{(2r)}\subset \calm^r\cali^r$ for all $r\geq 1$.
\end{theorem}
\proof
   Let $D=m_0(m_0+1)+s\cdot m(m+1)+\rho(m)$.\\
   If $m_0\geq\frac{m_0+sm}{\sqrt{s+2}}$ we put $R=2m_0$.
   Then part a) of Lemma \ref{lem:gwiazdka} implies
   \begin{equation}\label{eq:sss}
      R^2-3R\geq D
   \end{equation}
   in this case.\\
   If $m_0\leq\frac{m_0+sm}{\sqrt{s+2}}$ we put $R=\frac{2(m_0+sm)}{\sqrt{s+2}}$.
   Then part b) in Lemma \ref{lem:gwiazdka} implies again
   \eqnref{eq:sss} in this case.

   Thus \eqnref{eq:sss} is always satisfied and
   Lemma \ref{lem:gwiazdka} implies that there exists $d$
   satisfying \eqnref{eqreg} and \eqnref{eqalpha}. If this $d$
   satisfies also \eqnref{eq:m12}, then we are done.

   Otherwise we work with a new $d':=m_1+m_2>d$. This $d'$
   satisfies obviously \eqnref{eq:m12} and also
   \eqnref{eqreg} and \eqnref{eqalpha} as $d'+2\leq m_1+m_2+m_3+m_4$.
   Thus Theorem \ref{zastosowanie} gives the assertion.
\endproof

\subsection{Proof of Theorem B b)}
   We work under the assumption
   $$m_s\geq \frac{m_1}{2}.$$
   In the sequel we will use again the following notation:
   $$\Sigma=m_1+\ldots+m_s,\;\mbox{ and }\;
      Q=m_1^2+\ldots+m_s^2.$$

   Aiming for Theorem \ref{result} we prove first the following
   purely numerical lemma:
\begin{lemma}\label{comb1}
   With the above notation and assuming $s\geq 9$ we have
   \begin{equation}\label{eq:mmm}
      \frac{4}{s+1}\Sigma^2 \geq Q+\left(\frac{6}{\sqrt{s+1}}+1\right)\Sigma+\varrho(m_4).
   \end{equation}
   except for $m_1=2$, $m_2 = \ldots = m_s = 1$.
\end{lemma}
\proof
   We relax the assumption that $m_i$'s are integers
   (but we keep the monotonicity assumption $m_1\geq m_2 ... \geq m_s\geq \frac{m_1}{2}$).
   Working
   with fixed $m_1$, $m_2$, $m_3$, $m_4$ and $\Sigma$ we adjust all
   other $m_i$'s so that the right hand side of \eqnref{eq:mmm}
   becomes maximal possible.
   Observe that if we subtract from an $m_k$ a positive number $a$ and we add this number to a bigger $m_j$ with
   $j>4$, $\Sigma$ will remain the same and $Q$ will become
   bigger.
   We proceed
   this way to get the sequence of $m_i$'s satisfying the assumed conditions and maximalizing the right
   hand side. The sequence becomes:
   $$m_1, m_2, m_3, m_4=m_5=\dots=m_t, m_{t+1}, \frac{m_1}{2},\dots,\frac{m_1}{2}\;\mbox{ with }\; t\geq 4$$
   where $\frac{m_1}{2}$ stands on positions $t+2,\dots,s$ and we have $\frac{m_1}{2}\leq m_{t+1}\leq m_4$.

   Now, the idea is to reduce inequality \eqnref{eq:mmm} to
   inequality \eqnref{eq:nowa} below, which is easier to prove
   later. To alleviate the notation we put $x:=m_1$ and $y:=m_4$,

   First we bound the left hand side of the inequality \eqnref{eq:mmm} from below:
   $$\frac{4}{s+1}\Sigma^2\geq
     \frac{4}{s+1}\left(x+(t-1)y+(s-t)\frac{x}{2}\right)^2
     =\frac{4}{s+1}\left((t-1)y+(s-t+2)\frac{x}{2}\right)^2.$$

   The right hand side of the inequality \eqnref{eq:mmm} can in turn be bounded from above replacing $m_2$ and $m_3$ by
   $x$,
   $m_{t+1}$ by $y$ and finally $\varrho(y)$ by $3y^2$ as
   follows:

   $$Q+\left(\frac{6}{\sqrt{s+1}}+1\right)\Sigma+\varrho(y)\leq $$
   $$\leq  \left( 3x^2+(t-2)y^2+(s-t-1)\frac{x^2}{4}\right)+
   $$$$+\left(\frac{6}{\sqrt{s+1}}+1\right)\left(
     3x+(t-2)y+(s-t-1)\frac{x}{2}\right)+3y^2=$$
   $$=\left((t+1)y^2+(s-t+11)\frac{x^2}{4}\right)+\left(\frac{6}{\sqrt{s+1}}+1\right)\left(
     (t-2)y+(s-t+5)\frac{x}{2}\right).$$

   Thus \eqnref{eq:mmm}, will follow from
   \begin{equation}\label{eq:nowa}
   \begin{split}
      4\left((t-1)y+(s-t+2)\frac{x}{2}\right)^2 \geq\\
      (s+1)\left((t+1)y^2+(s-t+11)\frac{x^2}{4}\right)+\\
      \left(\frac{6}{\sqrt{s+1}}+1\right)\left(
      (t-2)y+(s-t+5)\frac{x}{2}\right).
   \end{split}
   \end{equation}

   Bounding $\frac{6}{\sqrt{s+1}}$ from above by $2$, multiplying by $4$ and bringing everything on
   one side, we get a new inequality which implies
   \eqnref{eq:nowa} of course.

   \begin{equation}\label{eq:nowa2}
   \begin{array}{c}
   (3x^2-6x)s^2+
   (16xy-7x^2-4y^2+6x-12y)st+
   4(x^2+4y^2-4xy)t^2+\\
   4(x^2-y^2-4xy-9x+6y)s+
   3(16xy-5x^2-12y^2+2x-4y)t+\\
   5x^2+12y^2-32xy+24y-30x\geq 0
   \end{array}
   \end{equation}
   with constraints
   \begin{equation}\label{eq:warunki}
      x\geq y, \ 2y\geq x, \ s\geq 9, \ t\geq 4\; \mbox{ and }\; s\geq t.
   \end{equation}

   Now, we prove inequality \eqnref{eq:nowa2} under additional
   assumption $x\geq 4$.
   It can be easily checked that the coefficients at $t^2$ and $st$
   are always positive under conditions \eqnref{eq:warunki}.

   Similarly we have positivity for slightly modified (by the underlined terms)
   coefficients
   at $s$, $t$ and the constant term
   \begin{align*}
   \underline{19x^2}+4(x^2-y^2-4xy-9x+6y)\geq 0, \\
   \underline{\frac{9}{2}x^2}+3(16xy-5x^2-12y^2+2x-4y)\geq 0,\\
   \underline{\frac{33}{2}x^2}+5x^2+12y^2-32xy+24y-30x\geq 0.
   \end{align*}

   The proof will be finished, once we show that the inequality
   \begin{equation}\label{eq:last}
      6x^2s^2\geq 12xs^2+38x^2s+9x^2t+33x^2
   \end{equation}
   is satisfied under conditions \eqnref{eq:warunki}. Using the
   estimate $t\leq s$ and dividing by $x$ this is reduced to
   \begin{equation}\label{eq:last2}
      6xs^2\geq 12s^2+(38xs+9xs+33x)
   \end{equation}

   For $s\geq 17$, the term $3xs^2$ bounds both summands on the right
   hand side in \eqnref{eq:last2} for all $x\geq 4$.

   Similarly, for $x\geq 34$ we can bound both summands in the
   following manner
   $$0.36xs^2\geq 12s^2\; \mbox{ and }\; (6-0.36)xs^2\geq
   38xs+9xs+33x$$
   for all $s\geq 9$.

   Thus we are left with finitely many cases, namely
   $$9\leq s\leq 16\;\mbox{ and }\; 4\leq x\leq 33$$
   for which we check \eqnref{eq:nowa2} directly by
   some dull computations omitted here.

   It remains to check \eqnref{eq:mmm} with cases $x=3$ and $x=2$
   which we leave to a motivated enough reader. Note that
   Lemma is not true in the case $x=2$ and $m_2=\dots=m_s=1$
   (which is covered by Proposition \ref{drugie}).
\endproof
   We are now in the position to finish the proof of part b) in Theorem
   B.

\begin{theorem}\label{result}
   Let $m_1\geq m_2 \geq ...\geq m_s\geq \frac{m_1}{2}$ be positive
   integers.
   We consider the ideal
   $\cali=\mathfrak{m}_{P_1}^{m_1}\cap\mathfrak{m}_{P_2}^{m_2}\cap...\cap\mathfrak{m}_{P_s}^{m_s}$
   of $s \geq 9$ general (fat) points. Then $\cali^{(2r)}\subset \calm^r\cali^r$
   for all $r\geq 1$.
\end{theorem}

\begin{proof}
   The case $m_1=\dots=m_s=1$ is covered by Theorem
   \ref{thm:hahu}.\\
   The case $m_1=2$ and $m_2=\dots m_s$ follows from Proposition
   \ref{drugie}.\\
   In the remaining cases the inequality \eqnref{eq:mmm}
   in Lemma \ref{comb1} is equivalent to inequality
   $$R^2-3R\geq D$$
   if we define
   $$D=\sum m_i(m_i+1)+\varrho(m_4)\;\mbox{ and }\;
   R=\frac{2}{\sqrt{s+1}}\sum m_i.$$
   By Lemma \ref{lem:gwiazdka} we know then that there exists an integer $d$
   satisfying \eqref{eqreg} and \eqref{eqalpha}. If $d \geq m_1+m_2$
   we are done by Theorem \ref{zastosowanie}. Otherwise, exactly as in the proof of Theorem
   \ref{result2},
   we take
   $d'=m_1+m_2$ and  we observe that  since
   $d'+2\leq m_1+m_2+m_3+m_4$, the integer
   $d'$ satisfies inequalities \eqref{eqreg} and \eqref{eqalpha}; so that
   Theorem
  \ref{zastosowanie} ends the proof.
\end{proof}




\bigskip \small

\bigskip
   Marcin Dumnicki,
   Jagiellonian University, Institute of Mathematics, {\L}ojasiewicza 6, PL-30-348 Krak\'ow, Poland

\nopagebreak
   \textit{E-mail address:} \texttt{Marcin.Dumnicki@im.uj.edu.pl}

\bigskip
   Tomasz Szemberg,
   Instytut Matematyki UP,
   Podchor\c a\.zych 2,
   PL-30-084 Krak\'ow, Poland.

\nopagebreak
   \textit{E-mail address:} \texttt{szemberg@up.krakow.pl}

\bigskip
   Halszka Tutaj-Gasi\'nska,
   Jagiellonian University, Institute of Mathematics, {\L}ojasiewicza 6, PL-30-348 Krak\'ow, Poland

\nopagebreak
   \textit{E-mail address:} \texttt{Halszka.Tutaj@im.uj.edu.pl}


\end{document}